\documentclass[12pt]{amsart}
\usepackage{latexsym}

\usepackage{amsmath}

\newcommand{\ignore}[1]{}

\newcommand{\hide}[1]{}

\DeclareMathOperator{\gexp}{\mathcal{E}}
\DeclareMathOperator{\Li}{Li}

\newcommand{\F}{\mathbb F}

\newcommand{\Z}[0]{\mathbb Z}

\newcommand{\Q}{\mathbb{Q}}

\newtheorem{dummy}{Dummy}

\newtheorem{lemma}[dummy]{Lemma}

\newtheorem{theorem}[dummy]{Theorem}

\theoremstyle{definition}

\theoremstyle{remark}

\newtheorem*{rem*}{Remark to ourselves}

\begin{document}

\bibliographystyle{amsalpha}

\author{Marina Avitabile}
\email{marina.avitabile@unimib.it}
\address{Dipartimento di Matematica e Applicazioni\\
  Universit\`a degli Studi di Milano - Bicocca\\
 via Cozzi 55\\
  I-20125 Milano\\
  Italy}
\author{Sandro Mattarei}
\email{smattarei@lincoln.ac.uk}
\address{Charlotte Scott Centre for Algebra\\
University of Lincoln \\
Brayford Pool
Lincoln, LN6 7TS\\
United Kingdom}

\title[Generalized finite polylogarithms]{Generalized finite polylogarithms}


\begin{abstract}
We introduce a generalization $\pounds_{d}^{(\alpha)}(X)$ of the finite polylogarithms
$\pounds_{d}^{(0)}(X)=\pounds_d(X)=\sum_{k=1}^{p-1}X^k/k^d$,  in characteristic $p$,
which depends on a parameter $\alpha$.
The special case $\pounds_{1}^{(\alpha)}(X)$ was previously investigated by the authors
as the inverse, in an appropriate sense, of a parametrized generalization of the truncated exponential
which is instrumental in a {\em grading switching} technique for non-associative algebras.
Here we extend such generalization to $\pounds_{d}^{(\alpha)}(X)$ in a natural manner,
and study some properties satisfied by those polynomials.
In particular, we find how the polynomials $\pounds_{d}^{(\alpha)}(X)$ are related to the powers of $\pounds_{1}^{(\alpha)}(X)$
and derive some consequences.
\end{abstract}

\maketitle

\section{Introduction}\label{sec:intro}

In current terminology and notation introduced in~\cite{EVG:polyanalogsI}, the \emph{finite polylogarithms} are the polynomials
$\pounds_d(X)=\sum_{k=1}^{p-1}X^k/k^d$, where $d$ is an integer,
conveniently and most interestingly viewed in prime characteristic $p$.
Although those polynomials, which are truncated versions of the series defining the classical polylogarithms,
were already introduced by Mirimanoff~\cite{Mirimanoff}
in his investigations on Fermat's Last Theorem, see~\cite[Lecture~VIII]{Ribenboim:13},
they have enjoyed renewed interest in recent years due to their connections with algebraic $K$-theory.

In this paper we introduce a parametrized generalization of the finite polylogarithms.
Our motivation stems from the occurrence of the special case $d=1$
as an appropriate compositional inverse of generalized exponentials expressed by certain Laguerre polynomials.
Those particular Laguerre polynomials were investigated by the authors in~\cite{AviMat:Laguerre}
as they play the role of generalized
exponentials in a \emph{grading switching} technique for modular, non-associative algebras, whose purpose is to
produce a new grading of an algebra from a given one.
We limit ourselves here to giving the definition and exponential-like property of those Laguerre polynomials,
referring the interested reader to a sketch of their role in grading switching in the Introduction of~\cite{AviMat:glog},
and full details of that application in~\cite{AviMat:Laguerre} and~\cite{AviMat:gradings}.

The Laguerre polynomials of interest here, regarded  as having coefficients
in the field $\F_{p}$ with $p$ elements, take the form
\[
L_{p-1}^{(\alpha)}(X)=(1-\alpha^{p-1})\sum_{k=0}^{p-1}
\frac{X^k}{(1+\alpha)(2+\alpha)\cdots (k+\alpha)} \in\F_p[\alpha,X],
\]
which specializes to the \emph{truncated exponential} $E(X)=\sum_{k=0}^{p-1}X^k/k!$ when we set
$\alpha=0$.
Their crucial property for the grading switching application is that they satisfy a congruence
which is a weak version of the fundamental functional equation
$\exp(X)\exp(Y)= \exp(X+Y)$ for the classical exponential series $\exp(X)=\sum_{k=0}^{\infty}X^{k}/k!$ in characteristic zero.
Roughly speaking, the congruence relates the product $L_{p-1}^{(\alpha)}(X)L_{p-1}^{(\beta)}(Y)$ with $L_{p-1}^{(\alpha+\beta)}(X+Y)$, the latter multiplied by a
polynomial in $\F_{p}(\alpha,\beta)[X,Y]$ whose most important feature in this context is
that all its terms have total degree multiple of $p$.
We quote that result from~\cite{AviMat:Laguerre} in Theorem~\ref{thm:functional-Laguerre-simplified},
and then supplement it with a more precise version, Theorem~\ref{thm:functional-genexp},
where we provide explicit expressions for the coefficients of that polynomial.
In order to provide a solid motivation for the particular generalization of finite polylogarithms that we intend to study here,
which is inferred from the special case $d=1$,
we devote the remainder of Section~\ref{sec:gexp} to proving that the exponential-like property described by Theorem~\ref{thm:functional-Laguerre-simplified}
essentially characterizes the Laguerre polynomials under consideration.
We formalize our conclusion in Theorem~\ref{thm:equivalence}.

Thinking of $L_{p-1}^{(\alpha)}(X)$ as an exponential-like polynomial suggests that an appropriate
compositional inverse $\pounds_{1}^{(\alpha)}(X)$ of
$L_{p-1}^{(\alpha)}(X)$ may be interpreted as a logarithm-like polynomial.
Such inverse was investigated in the paper~\cite{AviMat:glog}, where it was denoted by $G^{(\alpha)}(X)$.
However, to match the standard notation $\pounds_{1}(X)$ for the first finite polylogarithm
we set here $\pounds_{1}^{(\alpha)}(X)=-G^{(\alpha)}(X)$.
The precise statement for $\pounds_{1}^{(\alpha)}(X)$ being (essentially) a left compositional inverse of $L_{p-1}^{(\alpha)}(X)$ then reads as
$\pounds_{1}^{(\alpha)}(X)$ being the unique polynomial of degree less than $p$ in $\F_{p}(\alpha)[X]$ such that
\[
-\pounds_{1}^{(\alpha)}\bigl(L_{p-1}^{(\alpha)}(X)\bigr)\equiv X \pmod{X^p-(\alpha^p-\alpha)}.
\]

Before we give, in the next paragraph, an explicit description of the coefficients of $\pounds_{1}^{(\alpha)}(X)$,
we wish to further stress that the above congruence is really what motivates its definition
as a logarithm-like polynomial, as (essentially) the left inverse of the exponential-like polynomial $L_{p-1}^{(\alpha)}(X)$
(and also a right inverse with respect to an appropriate different modulus).
In turn, the exponential-like property of $L_{p-1}^{(\alpha)}(X)$ determines that polynomial uniquely up to natural variations,
as we mentioned above.
Finally, the modulus of the above congruence is also natural and forced upon us by the application to grading switching.
Altogether, this constitutes a strong support for this particular generalization of
$\pounds_{1}(X)=\pounds_{1}^{(0)}(X)$ that we consider here.
Setting $\alpha=0$
the above congruence becomes
$
-\pounds_{1}\bigl(E(X)\bigr)\equiv X \pmod{X^p},
$
which according to the functional equation $\pounds_1(1-X)=\pounds_1(X)$ (as polynomials in $\F_p[X]$)
results from $\log\bigl(\exp(X)\bigr)=X$ upon viewing it first modulo $X^p$ and then modulo $p$.
The details of this deduction are explained in the discussion following~\cite[Theorem~2]{AviMat:glog}.

It turns out that the coefficients of $\pounds_{1}^{(\alpha)}(X)$ can be explicitly described as follows.
For integers $0<k<p$ and $0<a<p$, we let
$p^{e(k,a)}$ be the highest power of $p$ which divides the product of binomial coefficients
$\prod_{s=1}^{k}\binom{sa}{a}$,
and set
$g_k(\alpha)=\prod_{0<a<p}(1+\alpha/a)^{-e(k,a)}$,
viewed as a rational function in $\F_p(\alpha)$.
Then $\pounds_1^{(\alpha)}(X)=\sum_{k=1}^{p-1} g_k(\alpha)\,X^k/k$.
This description of the coefficients $g_k(\alpha)$ of $\pounds_1^{(\alpha)}(X)$
is more compact than the original one we gave in~\cite[Subsection~2.2]{AviMat:glog}.
Most of the work to bring that description to the fully factorized and arguably more useful form given here was actually done in~\cite[Section~4]{AviMat:glog},
with a short supplementary argument which we provide in Subsection~\ref{subsec:gfp} of this paper.

To extend this generalization of $\pounds_1(X)$ to higher finite polylogarithms we note that
the various finite polylogarithms are connected one another by an application of the differential operator $X\,d/dX$.
If this rule is to be preserved in the generalization, it is natural to set
$\pounds_d^{(\alpha)}(X)=\sum_{k=1}^{p-1} g_k(\alpha)\,X^k/k^d$
for any integer $d$.
Of course $\pounds_{d+p-1}^{(\alpha)}(X)=\pounds_d^{(\alpha)}(X)$.
These polynomials in $\F_{p}(\alpha)[X]$, which generalize $\pounds_d(X)=\pounds_d^{(0)}(X)$, are the objects of interest in the remainder of the paper.

Functional equations for finite polylogarithms are of considerable interest, and we review some in Subsection~\ref{subsec:polylog}.
Some of them relate to a congruence which connects finite polylogarithms $\pounds_{d}(X)$ with powers of $\pounds_{1}(X)$, namely,
\[
\pounds_1(X)^d\equiv (-1)^{d-1} d!\, \pounds_d(1-X) \pmod {X^p},
\]
for $0<d<p-1$, which is Equation~\ref{eq:pounds_powers} below.
Our main result here is Theorem~\ref{thm:powers}, which gives an extension of this congruence
to our generalized finite polylogarithms $\pounds_d^{(\alpha)}(X)$.
In the generalized version of the congruence (which in our formulation rather extends the above after $X$ is substituted with $1-X$)
the right-hand side does not involve just $\pounds_d^{(\alpha)}(X)$ but is a linear combination of that and each lower one down to $\pounds_1^{(\alpha)}(X)$.
Finally, we deduce a couple of consequences from Theorem~\ref{thm:powers},
whose relevance we explain in Subsection~\ref{subsec:equations}.
In particular, our final result, Theorem~\ref{thm:GP}, gives an equation which expresses the finite polylogarithm $\pounds_d(X)$
as a linear combination of certain evaluations of all generalized finite polylogarithms $\pounds_{d}^{(r\alpha)}$
as $r$ varies from $1$ to $p-1$.
We collect all substantial proofs of our results on the generalized finite polylogarithms in the final Section~\ref{sec:proofs}.

\section{A generalized truncated exponential}\label{sec:gexp}

The classical (generalized) Laguerre polynomial of degree $n \geq 0$ is defined as
\[
L_n^{(\alpha)}(X)=\sum_{k=0}^n\binom{\alpha+n}{n-k}
\frac{(-X)^k}{k!},
\]
where $\alpha$ is a parameter, usually taken in the complex numbers.
However, we may also view $L_n^{(\alpha)}(X)$ as a polynomial with rational coefficients in the two
indeterminates $\alpha$ and $X$, hence in the polynomial ring $\Q[\alpha,X]$.

Having fixed a prime $p$, we are only interested in Laguerre polynomials of degree $n=p-1$, whose coefficients
are $p$-integral and can be viewed modulo $p$.
Throughout the paper we work directly in characteristic $p$ rather than over the rationals,
thus regarding $L_{p-1}^{(\alpha)}(X)$ as a polynomial in $\F_{p}[\alpha,X]$.
The explicit form for  $L_{p-1}^{(\alpha)}(X)$  mentioned in the introduction easily follows from the classical definition taking into
account the identities $k!(p-1-k)!=(-1)^{k-1}$ for $0\leq k<p$ and
$\alpha^{p-1}-1=\prod_{k=1}^{p-1}(\alpha+k)$ in $\F_{p}[\alpha]$. We quote from \cite{AviMat:Laguerre}
a congruence which we will use later
\begin{equation}\label{eq:L-diff}
X\frac{d}{dX}
L_{p-1}^{(\alpha)}(X)
\equiv
(X-\alpha)\cdot L_{p-1}^{(\alpha)}(X)
\pmod
{X^p-(\alpha^p-\alpha)},
\end{equation}
and that may be thought of as an analogue of the differential equation
$\exp^{\prime}(X)=\exp(X)$ for the  classical exponential series.
The differential equation for the polynomials $L_{p-1}^{(\alpha)}(X)$ stated in Equation~\ref{eq:L-diff} was used in~\cite{AviMat:Laguerre}
to prove the following analogue of the functional equation $\exp(X)\exp(Y)=\exp(X+Y)$.

\begin{theorem}[{\cite[Proposition~2]{AviMat:Laguerre}}]\label{thm:functional-Laguerre-simplified}
Let $\alpha,\beta,X,Y$ be indeterminates over $\F_p$.
There exist rational expressions $c_i(\alpha,\beta)\in\F_p(\alpha,\beta)$
such that
\[
L_{p-1}^{(\alpha)}(X)\cdot
L_{p-1}^{(\beta)}(Y)
\equiv
L_{p-1}^{(\alpha+\beta)}(X+Y)\cdot
\biggl(
c_0(\alpha,\beta)+\sum_{i=1}^{p-1}c_i(\alpha,\beta)X^iY^{p-i}
\biggr)
\]
in $\F_p(\alpha,\beta)[X,Y]$, modulo the ideal generated by
$X^p-(\alpha^p-\alpha)$
and
$Y^p-(\beta^p-\beta)$.
\end{theorem}

The actual statement of Proposition~2 in~\cite{AviMat:Laguerre} is stronger and more involved than Theorem~\ref{thm:functional-Laguerre-simplified},
as it had to provide a sharper control over the rational expressions $c_i(\alpha,\beta)$, which was required for an application to grading switching.
The expressions $c_i(\alpha,\beta)$ are actually uniquely determined, and are given by
$
c_0(\alpha,\beta)
=
-(\alpha-1)_{p-1}(\beta-1)_{p-1}/(\alpha+\beta-1)_{p-1},
$
and
$
c_i(\alpha,\beta)
=
-(\alpha-1)_{p-1-i}(\beta-1)_{i-1}/(\alpha+\beta-1)_{p-1}
$
for $0<i<p$.
These explicit formulas were omitted from~\cite{AviMat:Laguerre} as their available proof was awkward,
but they will now follow from Theorem~\ref{thm:functional-genexp} below.

A simplification in those formulas and their proof results from a natural normalization of our Laguerre polynomials
to turn their constant term into $1$:
\[
\gexp^{(\alpha)}(X)
:=\frac{L_{p-1}^{(\alpha)}(X)}{1-\alpha^{p-1}}
=\sum_{k=0}^{p-1}
\frac{X^k}{(1+\alpha)(2+\alpha)\cdots (k+\alpha)} \in\F_p(\alpha)[X].
\]
While $L_{p-1}^{(\alpha)}(X)$ has the advantage of having polynomial coefficients in $\alpha$,
which was a mild simplification in its application to grading switching in~\cite{AviMat:Laguerre},
the polynomial $\gexp^{(\alpha)}(X)$ seems a more natural analogue of the exponential function.
We now prove a more precise version of Theorem~\ref{thm:functional-Laguerre-simplified} in terms of $\gexp^{(\alpha)}(X)$,
where the coefficients are given explicitly.

\begin{theorem}\label{thm:functional-genexp}
Let $\alpha,\beta,X,Y$ be indeterminates over $\F_p$.
Then
\[
\gexp^{(\alpha)}(X)\cdot
\gexp^{(\beta)}(Y)
\equiv
\gexp^{(\alpha+\beta)}(X+Y)\cdot
\biggl(
1+\sum_{i=1}^{p-1}\frac{X^iY^{p-i}}
{(\alpha+i)_{i}\,(\beta+p-i)_{p-i}}
\biggr)
\]
in $\F_p(\alpha,\beta)[X,Y]$, modulo the ideal generated by
$X^p-(\alpha^p-\alpha)$
and
$Y^p-(\beta^p-\beta)$.
\end{theorem}

\begin{proof}
We know from Theorem~\ref{thm:functional-Laguerre-simplified} that there exist rational expressions $s_i(\alpha,\beta)\in\F_p(\alpha,\beta)$ such that
\[
\gexp^{(\alpha)}(X)\cdot
\gexp^{(\beta)}(Y)
\equiv
\gexp^{(\alpha+\beta)}(X+Y)\cdot
\biggl(
s_0(\alpha,\beta)+\sum_{i=1}^{p-1}s_i(\alpha,\beta)X^iY^{p-i}
\biggr)
\]
in $\F_p(\alpha,\beta)[X,Y]$, modulo the ideal generated by
$X^p-(\alpha^p-\alpha)$
and
$Y^p-(\beta^p-\beta)$.

It will turn out that the expressions $s_i(\alpha,\beta)$ are actually uniquely determined,
and we will compute them by comparing coefficients of certain monomials in both sides of the above congruence, after reduction by the moduli.
First, the only term in the product at the right-hand side of the congruence
in which both exponents of $X$ and $Y$ are multiples of $p$ is
$s_0(\alpha,\beta)$,
hence comparing constant terms in both sides of the congruence we find
$s_0(\alpha,\beta)=1$.

Now compare the coefficients of $X^k$ in both sides of the congruence for $0<k<p$.
In the left-hand side that coefficient equals $1/(\alpha+k)_k$.
In the right-hand side, after reducing modulo $Y^p-(\beta^p-\beta)$ the coefficient of $X^k$ equals
\[
\frac{1}{(\alpha+\beta+k)_k}
+\frac{\beta^p-\beta}{(\alpha+\beta+k)_{k}}
\sum_{i=1}^k\binom{k}{i}
s_i(\alpha,\beta).
\]
Consequently, we find
\[
(\beta^p-\beta)\sum_{i=1}^k\binom{k}{i}
s_i(\alpha,\beta)
=
\frac{(\alpha+\beta+k)_k}{(\alpha+k)_k}-1
=
\frac{1}{(\alpha+k)_k}
\sum_{i=1}^k\binom{k}{i}
(\alpha+k)_{k-i}(\beta)_i,
\]
where we have applied the {\em binomial theorem} for falling factorials, and hence
\[
\sum_{i=1}^k\binom{k}{i}
s_i(\alpha,\beta)
=
\sum_{i=1}^k\binom{k}{i}
\frac{1}{(\alpha+i)_i(\beta+p-i)_{p-i}}.
\]
This yields
\[
s_i(\alpha,\beta)
=\frac{1}{(\alpha+i)_{i}(\beta+p-i)_{p-i}}
\]
for $0<i<p$, as desired.
\end{proof}

The special case of Theorem~\ref{thm:functional-genexp} where $\alpha=\beta=0$ concerns the truncated exponential $\gexp^{(0)}(X)=E(X)$
and is~\cite[Proposition~1]{AviMat:gradings}, noting that
$(i)_{i}\,(p-i)_{p-i}=i!(p-i)!\equiv(-1)^i i\pmod{p}$.

As we mentioned in Section~\ref{sec:intro}, the existence of a congruence as in Theorem~\ref{thm:functional-Laguerre-simplified},
for some unspecified rational expressions $c_i(\alpha,\beta)$,
suffices to characterize the polynomials $L_{p-1}^{(\alpha)}(X)$ among the polynomials in $\F_{p}[\alpha][X]$, up to some natural variations.
For convenience, we rather state and prove an essentially equivalent characterization of their scalar multiples $\mathcal{E}^{(\alpha)}(X)$,
among the polynomials in $\F_{p}(\alpha)[X]$, again up to some natural variations.

\begin{theorem}\label{thm:equivalence}
Let $\alpha, \beta, X, Y$ be indeterminates over $\F_{p}$ and let $P^{(\alpha)}(X)$ be a nonzero polynomial in $\F_{p}(\alpha)[X]$,
of degree less than $p$.
Suppose that there exist rational expressions $s_i(\alpha,\beta) \in \F_{p}(\alpha,\beta)$ such that
\begin{equation}\label{eq:P^aP^b}
P^{(\alpha)}(X)\cdot
P^{(\beta)}(Y)
\equiv
P^{(\alpha+\beta)}(X+Y)\cdot
\Bigl(
1+\sum_{i=1}^{p-1}s_i(\alpha,\beta)X^iY^{p-i}
\Bigr)
\end{equation}
in $\F_{p}(\alpha,\beta)[X,Y]$ modulo the ideal generated by
$X^p-(\alpha^p-\alpha)$
and
$Y^p-(\beta^p-\beta)$.

Assume that none of the denominators of the expressions $s_i(\alpha,\beta)$ has $\beta$ as a factor, so $s_i(\alpha,0)$ are defined.
Assume also that $0$ is not a pole of $s_{p-1}(\alpha,0)$, nor of any coefficient of $P^{(\alpha)}(X)$,
so $s_{p-1}(0,0)$ and $P^{(0)}(X)$ are defined.

Then
$P^{(\alpha)}(X)=\gexp^{(c\alpha)}(cX)$ for some $c \in \F_{p}$.
\end{theorem}

To avoid obscuring the argument of the proof, we have placed various assumptions in Theorem~\ref{thm:equivalence}
on the denominators of the expressions $s_i(\alpha,\beta)$ and also of the coefficients of $P^{(\alpha)}(X)$.
In another version of this result one may take $P^{(\alpha)}(X)\in\F_p[\alpha][X]$,
hence with polynomial coefficients, rather than $P^{(\alpha)}(X)\in\F_p(\alpha)[X]$,
provided that one allows a further rational expression $s_{0}(\alpha,\beta)$ in place of the term $1$ in the right-hand side of the congruence.
Then quite similar arguments as in the proof of Theorem~\ref{thm:equivalence} show that
$P^{(\alpha)}(X)=d(\alpha)\cdot L_{p-1}^{(c\alpha)}(cX)$, for some polynomial $d(\alpha)\in\F_p[\alpha]$, and some $c \in \F_p$.

\begin{proof}
The polynomial $P^{(\alpha)}(X)$ must have a nonzero constant term $P^{(\alpha)}(0)$.
In fact, upon setting $Y=0$ and $\beta=0$, which is allowed according to our assumptions on the rational expressions $s_i(\alpha,\beta)$ and on $P^{(\alpha)}(X)$,
Equation~\eqref{eq:P^aP^b} yields
$
P^{(\alpha)}(X)\cdot
P^{(0)}(0)
\equiv
P^{(\alpha)}(X)
$
modulo $X^p-(\alpha^p-\alpha)$,
whence $P^{(0)}(0)=1$.

Because the only term in the product at the right-hand side of Equation~\eqref{eq:P^aP^b}
in which both exponents of $X$ and $Y$ are multiples of $p$ is the constant term $P^{(\alpha+\beta)}(0)$, we have
$
P^{(\alpha)}(0)\cdot
P^{(\beta)}(0)
=
P^{(\alpha+\beta)}(0)
$.
Setting $\beta=k\alpha$ we find
$
P^{(\alpha)}(0)\cdot
P^{(k\alpha)}(0)
=
P^{((k+1)\alpha)}(0)
$,
and working inductively we find
$
P^{(\alpha)}(0)^p
=
P^{(p\alpha)}(0)
=
P^{(0)}(0)
=1
$,
whence $P^{(\alpha)}(0)=1$.

Following a standard approach to functional equations such as Equation~\eqref{eq:P^aP^b} we apply the differential operator $d/dY$ to both sides.
This is allowed for the congruence because $d/dY$ annihilates both
$X^p-(\alpha^p-\alpha)$ and $Y^p-(\beta^p-\beta)$, and hence leaves invariant the ideal of $\F_{p}(\alpha,\beta)[X,Y]$ which they generate.
Multiplying the resulting congruence by $X$, and specializing $Y=0$ and $\beta=0$,
we find
\[
X\,P^{(\alpha)}(X)\cdot
c
\equiv
X\,\frac{dP^{(\alpha)}(X)}{dX}
+s_{p-1}(\alpha,0)P^{(\alpha)}(X)X^p
\pmod {X^p-(\alpha^p-\alpha)},
\]
where $c\in\F_p$ is the value of $dP^{(0)}(X)/dX$ at $X=0$.
After reducing by the modulus and rearranging terms this becomes
\[
X\,\frac{dP^{(\alpha)}(X)}{dX}
\equiv\bigl(cX-r(\alpha)\bigr)P^{(\alpha)}(X)
\pmod {X^p-(\alpha^p-\alpha)},
\]
where we have used the shorthand $r(\alpha)=(\alpha^p-\alpha)s_{p-1}(\alpha,0)$.
Note that $\alpha=0$ is a zero of $r(\alpha)$, otherwise it would be a pole of $s_{p-1}(\alpha,0)=r(\alpha)/(\alpha^p-\alpha)$,
contrary to one of our assumptions.
In particular, $r(\alpha)$ cannot be a nonzero constant.

If $c=0$ then both sides of the congruence are polynomials of degree less than $p$, hence the congruence is actually an equality,
and because $X\,dX^k/dX=kX^k$ it follows that $r(\alpha)=0$ and $P^{(\alpha)}(X)=1=\gexp^{(0\alpha)}(0X)$.

Now assume that $c\neq 0$ and write
$P^{(\alpha)}(X)=\sum_{k=0}^{p-1}c_k(\alpha)X^k$,
hence with $c_0(\alpha)=1$, and $c_1(0)=c$.
After expanding the right-hand side and replacing the term
$cX\cdot c_{p-1}(\alpha)X^{p-1}$ with
$c\cdot c_{p-1}(\alpha)\cdot(\alpha^p-\alpha)$,
the congruence becomes an equality as both sides have now degree less than $p$.
Equating term by term we find
\[
\begin{cases}
r(\alpha)=c\cdot c_{p-1}(\alpha)\cdot(\alpha^p-\alpha),
&\text{and}
\vspace{0.2cm}\\
\bigl(r(\alpha)+k\bigr)\cdot c_k(\alpha)=c\cdot c_{k-1}(\alpha)
&\text{for $1\leq k \leq p-1$}.
\end{cases}
\]
Because $r(\alpha)$ is not a nonzero constant, $r(\alpha)+k$ is never zero, and consequently none of the $c_k(\alpha)$ are zero.

As a preliminary step in solving this system for the rational expressions $c_k(\alpha)$
we note that the product of all $p$ equations reads
\[
\bigl(r(\alpha)^p-r(\alpha)\bigr)
\prod_{k=1}^{p-1}c_k(\alpha)
=
(\alpha^p-\alpha)c^p
\prod_{k=1}^{p-1}c_k(\alpha).
\]
Because $c^p=c$ this implies
$
\bigl(r(\alpha)-c\alpha\bigr)^p
=r(\alpha)-c\alpha
$,
whence $r(\alpha)-c\alpha\in\F_p$.
Because $r(\alpha)=0$ we deduce $r(\alpha)=c\alpha$.
Solving
\[
\begin{cases}
\alpha=c_{p-1}(\alpha)\cdot(\alpha^p-\alpha),
&\text{and}
\vspace{0.2cm}\\
\bigl(c\alpha+k\bigr)\cdot c_k(\alpha)=c\cdot c_{k-1}(\alpha)
&\text{for $1\leq k \leq p-1$},
\end{cases}
\]
we conclude that
$c_k(\alpha)=c^k/(c\alpha+k)_k$
for $0\le k<p$,
whence
$P^{(\alpha)}(X)=\gexp^{(c\alpha)}(cX)$ as desired.
\end{proof}

We should mention that an earlier special version of the grading switching achieved in~\cite{AviMat:Laguerre} through
the Laguerre polynomials $L_{p-1}^{(\alpha)}(X)$ was devised in~\cite{Mat:Artin-Hasse} using the
{\em Artin-Hasse} exponential series $E_p(X)=\prod_{i=0}^{\infty}\exp(X^{p^i}/p^i)$.
The coefficients of $E_p(X)$ are $p$-integral rational numbers and can therefore be viewed modulo $p$,
so one may regard $E_p(X)\in\F_p[[X]]$ for the sake of its application to grading switching.
The connection of the earlier theory based on the power series $E_p(X)$ with the more general one based on the polynomials $L_{p-1}^{(\alpha)}(X)$
is explained in~\cite[Proposition~6]{AviMat:gradings}, but here we stress that the success of the former crucially depended on a property of $E_p(X)$
analogous to the property of $L_{p-1}^{(\alpha)}(X)$ described in Theorem~\ref{thm:functional-Laguerre-simplified}:
each term of the power series $E_p(X)E_p(Y)/E_p(X+Y)\in\F_p[[X,Y]]$ has total degree a multiple of $p$.
It was then shown in~\cite{Mat:exponential} that this weak functional equation actually
characterizes $E_p(X)$ in the power series ring $F_p[[X]]$ up to certain natural variations.
Theorem~\ref{thm:equivalence} matches that result for the Laguerre polynomials $L_{p-1}^{(\alpha)}(X)$,
or their scalar multiples $\mathcal{E}^{(\alpha)}(X)$.

\section{Parametric versions of finite polylogarithms}\label{sec:gfp}

The finite polylogarithms
$\pounds_d(X)=\sum_{k=1}^{p-1} X^k/k^d$
are polynomial versions of the power series representations of the ordinary polylogarithms
$\Li_d(X)=\sum_{k=1}^{\infty} X^k/k^d$,
truncated as to make sense over a field of prime characteristic $p$.
In this section we extend the definition of finite polylogarithms
to include a parameter $\alpha$, motivated by the case $d=1$ which we extensively investigated in~\cite{AviMat:glog}.

\subsection{Some properties of finite polylogarithms}\label{subsec:polylog}

Before introducing our generalization $\pounds_d^{(\alpha)}(X)$
we discuss some of the remarkable properties of the finite polylogarithms $\pounds_d(X)$,
including some which we aim to extend to our parametrized versions.
Like their ordinary counterparts $\Li_d(X)$,
finite polylogarithms satisfy a number of functional equations, which are more abundant for small positive values of $d$.
In particular, $\pounds_1(X)$, which is a truncated version of the power series for $-\log(1-X)$
satisfies $\pounds_1(X)=-X^p \cdot \pounds_1(1/X)$ and
\begin{equation}\label{eq:2-term}
\pounds_{1}(X)=\pounds_{1}(1-X).
\end{equation}
Alternate application of these two equations yields six different equivalent representations for $\pounds_1(X)$,
see~\cite[Subsection~2.4]{AviMat:glog} or~\cite[Section~6]{MatTau:truncation} for broader discussions.
Those equations for $\pounds_1(X)$ do not appear to directly relate to any properties of the logarithmic function (or series),
but there is a two-variable functional equation which does, namely the {\em 4-term relation}
\begin{equation}\label{eq:4-term}
\pounds_{1}(X)-\pounds_{1}(Y)+X^p\pounds_{1}\left(\frac{Y}{X}\right)+(1-X)^{p}
\pounds_{1}\left(\frac{1-Y}{1-X}\right)=0,
\end{equation}
to be viewed as an identity in the polynomial ring $\F_p[X,Y]$.
In fact, it is possible to view this equation as an analogue of the classical equation $\log(xy)=\log(x)+\log(y)$,
in its equivalent form
$-\log(1-X)-\log(1-Y)=\log\bigl((1-Y)/(1-X)\bigr)$
in the power series ring $\F_p[[X,Y]]$,
and actually derive it from that.
See~\cite[Subsection~2.4]{AviMat:glog} for a sketch of an argument,
and~\cite[Section~6]{MatTau:truncation} for two different full proofs of Equation~\eqref{eq:4-term}
following this route.

A deeper connection between finite and ordinary polylogarithms was established by Elbaz-Vincent and Gangl in~\cite{EVG:polyanalogsI},
stimulated by questions raised by Kontsevich~\cite{Kontsevich}, who had first exhibited a version of Equation~\eqref{eq:4-term}
dubbing it the {\em generalized fundamental equation of information theory}.
According to~\cite{EVG:polyanalogsI}, many known functional equations for $\pounds_d(X)$ are closely related to functional equations
for the ordinary polylogarithms $\Li_{d+1}(X)$ (with index raised by one),
and can be derived from the latter though a sort of {\em differential,} or {\em infinitesimal} process.
In particular, Equations~\eqref{eq:2-term} and~\eqref{eq:4-term} originate from functional equations for the dilogarithm $\Li_2(X)$,
see~\cite[Proposition~5.9]{EVG:polyanalogsI}.
The same connection works for the functional equations which we are about to discuss, namely the only functional equations which exist for arbitrary $d$.

One functional equation valid for every $\pounds_d(X)$ is the simple {\em inversion relation}~\cite[Proposition~5.7(1)]{EVG:polyanalogsI},
\begin{equation}\label{eq:pounds_inversion}
\pounds_d(X)=(-1)^d X^p \cdot \pounds_d(1/X)
\end{equation}
in $\F_p[X]$, whose special case $d=1$ we have already mentioned.
This is an immediate consequence of Wilson's theorem, $(p-1)!\equiv -1$ in $\F_p$,
and says that the polynomials $\pounds_d(X)$ are essentially self-reciprocal.
The only other functional equation for $\pounds_d(X)$ which exists for arbitrary $d$
is the {\em distribution relation}~\cite[Proposition~5.7(2)]{EVG:polyanalogsI},
\begin{equation}\label{eq:pounds_distribution}
\pounds_d(X^h)=h^{d-1}\sum_{j=0}^{|h|-1}
\frac{1-X^{ph}}{1-\omega^{pj}X^p}
\pounds_d(\omega^j X),
\end{equation}
where $\omega$ is a a primitive $h$th root of unity.
This formulation of the distribution relation restricts the integer $h$ not to be a multiple of $p$, and Equation~\eqref{eq:pounds_distribution} formally takes place
in $\F_q[X]$ for some finite field extension of $\F_q$ containing such a root of unity,
or in fact in its quotient field $\F_q(X)$ when $h$ is negative.
(This restriction could be avoided by viewing the distribution relation as a congruence over a suitable number field rather than an equation over $\F_p$.)
As pointed out in~\cite{EVG:polyanalogsI}, Equation~\eqref{eq:pounds_inversion}
may be viewed as the special case $h=-1$ of Equation~\eqref{eq:pounds_distribution}.

When we view the distribution relation modulo $X^p-1$ all summands vanish except for that with $j=0$, and we find
\begin{equation}\label{eq:pounds_d(x^h)}
\pounds_d(X^h)\equiv h^d \pounds_d(X)\pmod{X^p-1}
\end{equation}
in $\F_p[X]$, again for $h$ not a multiple of $p$.
Replacing $X$ with $1-X$ we can rewrite this in the equivalent form
$
\pounds_d\bigl((1-X)^h\bigr)\equiv h^d \pounds_d(1-X)\pmod{X^p}.
$
In the special case where $d=1$ this can be viewed as a congruence version of the property $\log(x^h)=h\,\log(x)$ of the logarithm.

Equation~\eqref{eq:pounds_d(x^h)} can also be lifted from its special case $d=1$ by means of
a congruence relating finite polylogarithms $\pounds_{d}(X)$ to powers of $\pounds_{1}(X)$, namely,
\begin{equation}\label{eq:pounds_powers}
\pounds_1(X)^d\equiv (-1)^{d-1} d!\, \pounds_d(1-X) \pmod {X^p},
\end{equation}
for $0<d<p-1$.
This congruence, as well as much of the material on finite polylogarithms reviewed here, traces back to Mirimanoff~\cite{Mirimanoff},
who developed it in his investigations on Fermat's Last Theorem, see~\cite[Lecture~VIII, Equation~(1,.27)]{Ribenboim:13}.
A modern proof of a slightly sharper version modulo $X^{p+1}$ of Equation~\eqref{eq:pounds_powers},
which involves a Bernoulli number, can be found in~\cite[Lemma~3.2]{MatTau:polylog}.
When $d=1$ Equation~\eqref{eq:pounds_powers} is a consequence of Equation~\eqref{eq:2-term},
and when $d=2$ or $3$ it can be strengthened to exact functional equations (meaning equalities, not just congruences)
by adding suitable extra terms, see~\cite[Equations~(14) and~(15)]{MatTau:polylog}, also already known to Mirimanoff.
A way of deriving those functional equations for $d=2,3$ from Equation~\eqref{eq:pounds_powers} by the sole use of symmetries is given in~\cite[Section~3]{MatTau:polylog}.
However, no such refinement is known (or likely even exists) for larger values of $d$.

\subsection{Generalized finite polylogarithms}\label{subsec:gfp}

We recall our generalization $\pounds_d^{(\alpha)}(X)$ of finite polylogarithms which we anticipated in Section~\ref{sec:intro}.
For integers $0<k<p$ and $0<a<p$, we let
$p^{e(k,a)}$ be the highest power of $p$ which divides the product of binomial coefficients
$\prod_{s=1}^{k}\binom{sa}{a}$,
and set
$g_k(\alpha)=\prod_{0<a<p}(1+\alpha/a)^{-e(k,a)}$,
viewed as a rational function in $\F_p(\alpha)$.
Then for any integer $d$ we set
\[
\pounds_d^{(\alpha)}(X)=\sum_{k=1}^{p-1} g_k(\alpha)\,X^k/k^d.
\]
This definition has its roots in the special case $d=1$, where
$\pounds_{1}^{(\alpha)}(X)$ is a left compositional inverse of $L_{p-1}^{(\alpha)}(X)$
in the context of the previous section, namely it satisfies
\[
-\pounds_1^{(\alpha)}\bigl(L_{p-1}^{(\alpha)}(X)\bigr)\equiv X \pmod{X^p-(\alpha^p-\alpha)}.
\]
Thus, $\pounds_{1}^{(\alpha)}(X)$ serves a generalization of the truncated logarithm $\pounds_1(X)=\pounds_1^{(0)}(X)$
matching the way $L_{p-1}^{(\alpha)}(X)$ generalizes the truncated exponential.
This definition of $\pounds_1^{(\alpha)}(X)$ extends naturally to $\pounds_d^{(\alpha)}(X)$ by imposing that they have no constant term and they satisfy
$(X\,d/dX)\pounds_d^{(\alpha)}(X)=\pounds_{d-1}^{(\alpha)}(X)$ for all integers $d$,
which is the way ordinary truncated polylogarithms $\pounds_d(X)$ are related.
Because $\pounds_{d+p-1}^{(\alpha)}(X)= \pounds_{d}^{(\alpha)}(X)$, we can assume $0\le d<p-1$ in the sequel.
Also, the case of $p=2$ is uninteresting as then
$\pounds_d^{(\alpha)}(X)=X$ for all $d$, and so we assume $p$ odd throughout this section.

The coefficients $g_k(\alpha)$ originally arose in~\cite{AviMat:glog} as
$g_k(\alpha)=1/\prod_{s=1}^{k-1}b_{1,s}(\alpha)$,
with the polynomials $b_{1,s}(\alpha)\in\F_{p}[\alpha]$ defined as
\[
b_{1,s}(\alpha)=
\sum_{k=0}^{p-1}(-1/s)^k\binom{\alpha -1}{p-1-k}\binom{s \alpha-1}{k},
\]
for $0<s<p-1$.
As explained there they can be viewed as special values of certain Jacobi polynomials,
but what matters here are their full factorizations in $\F_p[\alpha]$, which were found in~\cite{AviMat:glog}.
According to~\cite[Lemma~11]{AviMat:glog} those polynomials satisfy
$b_{1,s}(\alpha)b_{1,s}(-\alpha)=1-\alpha^{p-1}$,
whence each has degree $(p-1)/2$,
which was not obvious from their definition as sums.
Furthermore, the equation implies that $b_{1,s}(\alpha)$ factorizes into products of distinct linear factors in $\F_p[\alpha]$,
and exactly one of each pair of opposite nonzero elements of $\F_p$ is a root.
A simple characterization of which elements of $\F_p$ are roots of $b_{1,s}(\alpha)$ was given in~\cite[Theorem~12]{AviMat:glog},
and for completeness we now show how that leads to the definition of the rational functions $g_k(\alpha)$ which we gave above.

\begin{lemma}
For $0<k<p$ we have
$g_k(\alpha)=1/\prod_{s=1}^{k-1}b_{1,s}(\alpha)$.
\end{lemma}

\begin{proof}
Each polynomial $b_{1,s}(\alpha)$ has constant term $b_{1,s}(0)=\sum_{k=0}^{p-1}(-1)^k(1/s)^k=1$, hence
its factorization in $\F_p[\alpha]$ is completely described by its roots.
According to~\cite[Theorem~12]{AviMat:glog}, in its alternate formulation given in~\cite[Remark~13]{AviMat:glog},
an integer $0<a<p$ is a root of $b_{1,s}(\alpha)$ (when interpreted as its image in $\F_p$)
precisely when $p$ does not divide the binomial coefficient $\binom{a+sa}{a}$.
Equivalently, $-a$ is a root of $b_{1,s}(\alpha)$
precisely when $p$ divides the binomial coefficient $\binom{a+sa}{a}$.
Now note that $p^2$ cannot divide $\binom{a+sa}{a}$ and the conclusion follows.
\end{proof}

\subsection{Congruential functional equations for $\pounds_d^{(\alpha)}(X)$}\label{subsec:equations}
The main goal of the remainder of this paper is to provide analogues for our
generalized finite polylogarithms $\pounds_d^{(\alpha)}(X)$
of some of the known relations among finite polylogarithms defined for generic $d$,
which we summarized in Subsection~\ref{subsec:polylog}.
Thus, besides the easy Equation~\eqref{eq:pounds_inversion} we will generalize
Equation~\eqref{eq:pounds_powers}, and then use that to generalize
Equation~\eqref{eq:pounds_d(x^h)}.
We will state our results here and prove them in the next section.

We assign a name to a polynomial which will occur repeatedly, namely,
\begin{equation}\label{eq:T}
T(X):=L_{p-1}^{(X^p)}(X^p-X)=\prod_{i=1}^{p-1}(1+X/i)^i,
\end{equation}
where the explicit factorization given was proved in~\cite[Lemma~1]{AviMat:Laguerre}.
This polynomial will occur in the modulus $X^p-T(\alpha)$ of various congruences involving $\pounds_d^{(\alpha)}(X)$,
but also, for example, in an expression for the highest coefficient of $\pounds_d^{(\alpha)}(X)$, because
$g_{p-1}(\alpha)=(1-\alpha^{p-1})/T(\alpha)$.
This was proved in~\cite[Corollary~16]{AviMat:glog}, but can also be easily shown directly from our definition of $g_{p-1}(\alpha)$,
as we explain now as an example of such evaluations.

According to Lucas' theorem on binomial coefficients modulo a prime, $p$ divides the factor $\binom{sa}{a}$ in our definition of $g_k(\alpha)$ precisely when
the (least nonnegative) remainder of dividing $(s-1)a$ by $p$ is not less than $p-a$.
In the case of $g_{p-1}(\alpha)$ the remainders of dividing $(s-1)a$ by $p$, for a given $a$ as $s$ ranges over $0<s<p$,
will take all values from $0$ to $p-1$ with the exception of $p-a$,
hence precisely $a-1$ of them will exceed $p-a$.
Therefore, we find
$g_{p-1}(\alpha)=\prod_{a=1}^{p-1}(1-\alpha/a)^{-a+1}$
as desired.

Another relation among the coefficients $g_k(\alpha)$ amounts to the symmetry relation
$b_{1,s}(\alpha)=b_{1,p-1-s}(\alpha)$
of~\cite[Corollary~14]{AviMat:glog}, for $0<s<p-1$.
Taken together, in terms of the polynomials $g_k(\alpha)$, those equations are equivalent to
\begin{equation}\label{eq:sym}
g_k(\alpha)\cdot g_{p-k}(\alpha)=g_{p-1}(\alpha),
\qquad\text{for $0<k<p$.}
\end{equation}
As a consequence of this symmetry, one has
\[
T(\alpha)\cdot\pounds_{d}^{(\alpha)}(X)=(-1)^d X^{p} \cdot \pounds_{d}^{(-\alpha)}
\biggl(\frac{1-\alpha^{p-1}}{X}\biggr),
\]
a generalization of Equation~\eqref{eq:pounds_inversion} which can be proved in the same way as its special case $d=1$ in~\cite[Theorem~6]{AviMat:glog}.

Our main result on generalized polylogarithms is a generalization of Equation~\eqref{eq:pounds_powers}.
This generalized version does not relate $\pounds_{1}^{(\alpha)}(X)^d$ to $\pounds_{d}^{(\alpha)}(X)$ alone,
but also involves lower polylogarithms.
Denoting by ${n \brack k}$ the (unsigned) Stirling number of the first kind, which for $0<k\le n$
may be characterized by the polynomial identities
$(X+n-1)_n=\sum_{k=1}^{n}{n \brack k}X^k$ in $\Z[X]$,
we have the following result.

\begin{theorem}\label{thm:powers}
For any $0<d< p-1$ we have
\[
\frac{\pounds_{1}^{(\alpha)}(X)^d}{d}\equiv
(-1)^{d-1}\sum_{r=0}^{d-1}{d \brack r+1} \alpha^r\pounds_{d-r}^{(\alpha)}(X) \pmod{X^p-T(\alpha)}
\]
in the polynomial ring $\F_{p}(\alpha)[X]$.
\end{theorem}

In its specialization at $\alpha=0$, because ${d \brack 1}=(d-1)!$ Theorem~\ref{thm:powers} reads
$\pounds_1 (X)^d \equiv (-1)^{d-1} d!\,\pounds_d(X)\pmod{X^{p}-1}$.
We recover Equation~\eqref{eq:pounds_powers}
by replacing $X$ with $1-X$ and taking equation $\pounds_1 (X)=\pounds_1 (1-X)$ into account.

The congruence of Theorem~\ref{thm:powers} can be extended to $d=p-1$
but requires an extra term $1-\alpha^{p-1}$ at the right-hand side in that case.
Because
${p-1\brack k}\equiv 1\pmod{p}$ for $0<k<p$,
the congruence for $d=p-1$ reads
\[
\pounds_{1}^{(\alpha)}(X)^{p-1}\equiv
1-\alpha^{p-1}+
\sum_{r=0}^{p-2}\alpha^r\pounds_{p-1-r}^{(\alpha)}(X) \pmod{X^p-T(\alpha)}.
\]

Our next result generalizes Equation~\eqref{eq:pounds_d(x^h)}.
Its special case where $d=1$ is~\cite[Theorem~8]{AviMat:glog}, and we use
Theorem~\ref{thm:powers} to extend that to higher values of $d$.

\begin{theorem}\label{thm:pounds_d(x^h)}
For $0<h<p$ and $0<d<p-1$ we have
\[
\pounds_{d}^{(h\alpha)}\bigl(g_h(\alpha)X^h\bigr)
\equiv h^d \pounds_{d}^{(\alpha)}(X)
\pmod {X^p-T(\alpha)}
\]
in the polynomial ring $\F_{p}(\alpha)[X]$.
\end{theorem}

Our final result combines evaluations of all generalized finite polylogarithms $\pounds_{d}^{(r\alpha)}$ as $r$ varies from $1$ to $p-1$
and relates them to the standard finite polylogarithm $\pounds_d(X)$.
To avoid having to extend the ground field with $\alpha^{1/p}$ we conveniently replace $\alpha$ with $\alpha^p$ in the statement.

\begin{theorem}\label{thm:GP}
For any integer $d$ we have
\[
\sum_{r=1}^{p-1}\pounds_{d}^{(r\alpha^p)}\bigl(T(r\alpha)\,X\bigr)=(\alpha^{p-1}-1)\pounds_d(X)
\]
in the polynomial ring $\F_{p}(\alpha)[X]$.
\end{theorem}

Note that Theorem~\ref{thm:GP} states an identity, not just a congruence.
Because $\pounds_{d}^{(0)}(X)=\pounds_{d}(X)$ and $T(0)=1$ that can also be written as
\[
\sum_{r=0}^{p-1}\pounds_{d}^{(r\alpha^p)}\bigl(T(r\alpha)\,X\bigr)=\alpha^{p-1}\pounds_d(X).
\]
In a sense the special case $d=1$ of Theorem~\ref{thm:GP} gives an admittedly rather trivial answer to Question~7 in~\cite{AviMat:glog},
which asked for a generalization of the functional equation $\pounds_1(1-X)=\pounds_1(X)$
for the polynomials $\pounds_1^{(\alpha)}(X)$, possibly involving various values of $\alpha$:
when $d=1$ the left-hand side of the equation of Theorem~\ref{thm:GP} is invariant under the substitution
$X\mapsto 1-X$, because the right-hand side is.
A subtler answer appears now unlikely.

\section{Proofs of Theorem~\ref{thm:powers}, Theorem~\ref{thm:pounds_d(x^h)}, and Theorem~\ref{thm:GP} }\label{sec:proofs}

Our proof of Theorem~\ref{thm:powers} will proceed by applying the differential operator $X\,d/dX$ to the desired congruence,
whence the left-hand side will give
$\pounds_0^{(\alpha)}(X)\cdot\pounds_1^{(\alpha)}(X)^{d-1}$.
Working inductively, a crucial step will be expressing the product of $\pounds_0^{(\alpha)}(X)$ and $\pounds_1^{(\alpha)}(X)$
as a linear combination of them, which is what the next congruence achieves.

\begin{lemma}\label{lemma:G_0G_1}
The product $\pounds_{0}^{(\alpha)}(X) \cdot \pounds_{1}^{(\alpha)}(X)$ satisfies the congruence
\[
\pounds_{0}^{(\alpha)}(X) \cdot \pounds_{1}^{(\alpha)}(X)
\equiv
-\pounds_{1}^{(\alpha)}(X)
-\alpha \pounds_{0}^{(\alpha)}(X)
\pmod{X^p-T(\alpha)}
\]
in the polynomial ring $\F_{p}(\alpha)[X]$.
\end{lemma}

\begin{proof}
We apply the differential operator $d/dX$ to both sides of the congruence
\begin{equation}\label{eq:inverses}
-\pounds_{1}^{(\alpha)}(L_{p-1}^{(\alpha)}(X))\equiv X
\pmod{X^p-(\alpha^p-\alpha)},
\end{equation}
using Equations~\eqref{eq:L-diff} and
$(X\,d/dX)\pounds_d^{(\alpha)}(X)=\pounds_{d-1}^{(\alpha)}(X)$.
Noting that both $X$ and $L_{p-1}^{(\alpha)}(X)$
are coprime with the modulus, the latter because of Equation~\eqref{eq:T},
and hence are invertible in the quotient ring $\F_p(\alpha)[X]/\bigl(X^p-(\alpha^p-\alpha)\bigr)$, we find
\[
-\frac{1}{L_{p-1}^{(\alpha)}(X)}
\pounds_{0}^{(\alpha)}\bigl(L_{p-1}^{(\alpha)}(X)\bigr)
\cdot
\frac{X-\alpha}{X}
L_{p-1}^{(\alpha)}(X)
\equiv 1
\pmod{X^p-(\alpha^p-\alpha)}.
\]
After cancellation and multiplication by $X$ we obtain
\[
-\pounds_{0}^{(\alpha)}\bigl(L_{p-1}^{(\alpha)}(X)\bigr)
\cdot
(X-\alpha)
\equiv X
\pmod{X^p-(\alpha^p-\alpha)}.
\]

Now we would like to regard $L_{p-1}^{(\alpha)}(X)$ as a new variable, but this will require a foray into a power series ring
in a similar fashion as in the proofs of Corollary~3 and Theorem~8 in~\cite{AviMat:glog}.
Thus, we extend the ground field to $\F_p(\alpha^{1/p})$,
where $X^p-(\alpha^p-\alpha)$ becomes a $p$th power, and after setting
$X=x+\alpha-\alpha^{1/p}$
the congruence we have found reads
\begin{equation}\label{eq:inverses_derived}
-\pounds_{0}^{(\alpha)}\bigl(L_{p-1}^{(\alpha)}(x+\alpha-\alpha^{1/p})\bigr)
\cdot
(x-\alpha^{1/p})
\equiv
x+\alpha-\alpha^{1/p}
\pmod{x^p},
\end{equation}
in the polynomial ring $\F_p(\alpha^{1/p})[x]$.
In the same way, Equation~\eqref{eq:inverses} is equivalent to the congruence
\begin{equation}\label{eq:inverses_power_series}
-\pounds_1^{(\alpha)}\bigl(L_{p-1}^{(\alpha)}(x+\alpha-\alpha^{1/p})\bigr)-(\alpha-\alpha^{1/p})\equiv x \pmod{x^p}
\end{equation}
in the polynomial ring $\F_p(\alpha^{1/p})[x]$.
However, both congruences can and will now be interpreted in the power series ring $\F_p(\alpha^{1/p})[[x]]$.

Because the polynomial $L_{p-1}^{(\alpha)}(x+\alpha-\alpha^{1/p})-\delta$, where $\delta=T(\alpha^{1/p})$,
has no constant term and a nonzero term of degree one,
it has a compositional inverse as a series in $\F_p(\alpha^{1/p})[[x]]$.
Therefore, in $\F_p(\alpha^{1/p})[[x]]$ we can apply the substitution (or change of uniformizing parameter)
\[
X=-L_{p-1}^{(\alpha)}(x+\alpha-\alpha^{1/p})+\delta
\]
(the symbol $X$ being reused now with a different meaning from earlier in the proof).
According to Equation~\eqref{eq:inverses_power_series} the inverse substitution satisfies
\[
x\equiv
-\pounds_{1}^{(\alpha)}(\delta-X)-(\alpha-\alpha^{1/p})
\pmod {X^p}.
\]
Applying this substitution to Equation~\eqref{eq:inverses_derived} we find
\[
-\pounds_{0}^{(\alpha)}(\delta-X)
\cdot
\bigl(-\pounds_{1}^{(\alpha)}(\delta-X)-\alpha\bigr)
\equiv
-\pounds_{1}^{(\alpha)}(\delta-X)
\pmod{X^p},
\]
Because this congruence involves only polynomials it actually takes place in the polynomial ring
$\F_p(\alpha^{1/p})[X]$.
Replacing $X$ with $\delta-X$ we have
\[
\pounds_{0}^{(\alpha)}(X)\bigl(\pounds_{1}^{(\alpha)}(X)+\alpha\bigr)\equiv -\pounds_{1}^{(\alpha)}(X)
\pmod{X^p-T(\alpha)},
\]
which is equivalent to the desired conclusion.
\end{proof}

We are now ready to present a proof of Theorem~\ref{thm:powers}.

\begin{proof}[Proof of Theorem~\ref{thm:powers}]
We will omit the modulus from all congruences in this proof, which will invariably be $X^p-T(\alpha)$.
We proceed by induction on $d$, the case $d=1$ being trivial, hence assume $d>1$.
Because
\[
\sum_{r=0}^{d-1}{d \brack r+1}(k\alpha)^{r}
=(k\alpha+d-1)_d/(k\alpha)
=(k\alpha+d-1)_{d-1},
\]
the desired conclusion can be written as
\[
\pounds_{1}^{(\alpha)}(X)^d/d\equiv
(-1)^{d-1} \sum_{k=1}^{p-1} (k\alpha+d-1)_{d-1}
\,g_k(\alpha)X^k/k^d.
\]
To prove this congruence, write
\[
\pounds_{1}^{(\alpha)}(X)^d/d\equiv \sum_{k=0}^{p-1}c_{k}(\alpha)X^k,
\]
as certainly holds for certain rational expressions $c_{k}(\alpha)\in \F_p(\alpha)$ to be determined.
Applying the differential operator $X\,d/dX$ to both sides of the congruence
we find
\[
\pounds_{0}^{(\alpha)}(X) \cdot \pounds_{1}^{(\alpha)}(X)^{d-1}\equiv
\sum_{k=1}^{p-1}k c_{k}(\alpha)X^k.
\]
Note that this kills the coefficient $c_0(\alpha)$, so we will deal with that separately later.
According to Lemma~\ref{lemma:G_0G_1} the above congruence is equivalent to
\[
-\pounds_{1}^{(\alpha)}(X)^{d-1}- \alpha \pounds_{0}^{(\alpha)}(X) \pounds_{1}^{(\alpha)}(X)^{d-2}
\equiv
\sum_{k=1}^{p-1}k c_{k}(\alpha)X^k.
\]
Now by the inductive hypothesis we have
\[
\pounds_{1}^{(\alpha)}(X)^{d-1}\equiv
(-1)^d(d-1) \sum_{k=1}^{p-1} (k\alpha+d-2)_{d-2}
\,g_k(\alpha)X^k/k^{d-1},
\]
and because
$\pounds_{0}^{(\alpha)}(X) \pounds_{1}^{(\alpha)}(X)^{d-2}$
results from applying the differential operator $X\,d/dX$ to
$\pounds_{1}^{(\alpha)}(X)^{d-1}/(d-1)$
we obtain
\[
\pounds_{0}^{(\alpha)}(X) \pounds_{1}^{(\alpha)}(X)^{d-2}\equiv
(-1)^d\sum_{k=1}^{p-1} (k\alpha+d-2)_{d-2}
\,g_k(\alpha)X^k/k^{d-2}.
\]
In conclusion, we find
\begin{align*}
k\,c_{k}(\alpha)
&=
(-1)^{d-1}(k\alpha+d-2)_{d-2}
(d-1+k\alpha)
\,g_k(\alpha)/k^{d-1}
\\&=
(-1)^{d-1}(k\alpha+d-1)_{d-1}
\,g_k(\alpha)/k^{d-1},
\end{align*}
for $1\leq k \leq p-1$.

In order to complete the proof it remains to show that $c_{0}(\alpha)$ vanishes.
Using the inductive hypothesis it suffices to show that
\[
\pounds_{1}^{(\alpha)}(X)\cdot
\sum_{k=1}^{p-1} (k\alpha+d-2)_{d-2}
\,g_k(\alpha)X^k/k^{d-1}
\]
has no term of degree $p$, as that is the only term which would contribute to $c_{0}(\alpha)$ after reduction modulo $X^p-T(\alpha)$.
In fact, the coefficient of $X^p$ in the above product equals
\[
\sum_{k=1}^{p-1}\frac{g_{p-k}(\alpha)}{p-k}
\cdot
\frac{g_{k}(\alpha)(k\alpha+d-2)_{d-2}}{k^{d-1}}=
g_{p-1}(\alpha)\sum_{k=1}^{p-1}\frac{(k\alpha+d-2)_{d-2}}{k^d},
\]
where we have used Equation~\eqref{eq:sym}.
The latter sum vanishes because $\sum_{k=1}^{p-1}1/k^r$ vanishes in $\F_p$ for $0<r<p-1$,
and $(k\alpha+d-2)_{d-2}$ has degree less than $p-1$ as polynomial in $k$.
\end{proof}

When $d=p-1$, a supplementary case which we mentioned after Theorem~\ref{thm:powers},
the inductive step extends in the above proof providing expressions for the coefficients $c_k(\alpha)$ for $0<k<p$, but the separate final argument for
the vanishing of $c_0(\alpha)$ needs modifications, and yields $c_0(\alpha)=\alpha^{p-1}-1$ instead.

The following proof of Theorem~\ref{thm:pounds_d(x^h)} relies on the special case where $d=1$, which is~\cite[Theorem~8]{AviMat:glog}, and uses
Theorem~\ref{thm:powers} to extend it to higher values of $d$.

\begin{proof}[Proof of Theorem~\ref{thm:pounds_d(x^h)}]
We proceed by induction on $d$, the case $d=1$ being~\cite[Theorem~8]{AviMat:glog}.
Hence assume $1<d<p-1$ and consider the right-hand side of the desired congruence multiplied by $d!$ to avoid introducing denominators.
According to Theorem~\ref{thm:powers},
\[
d!\, h^d \pounds_{d}^{(\alpha)}(X)
\equiv
(-1)^{d-1} (h \pounds_{1}^{(\alpha)}(X))^d-d\sum_{r=1}^{d-1}{d\brack r+1}
(h \alpha)^{r} h^{d-r}\pounds_{d-r}^{(\alpha)}(X)
\]
modulo $X^p-T(\alpha)$.
By induction we have
\[
h^{d-r} \pounds_{d-r}^{(\alpha)}(X)\equiv \pounds_{d-r}^{(h\alpha)}
\left(g_h(\alpha)X^h\right)
\pmod {X^p-T(\alpha)}
\]
for $0<r<d$,
and hence $d!\, h^d \pounds_{d}^{(\alpha)}(X)$ is congruent to
\[
(-1)^{d-1}\left(\pounds_{1}^{(h\alpha)}
\bigl(g_h(\alpha)X^h\bigr) \right)^d- d\sum_{r=1}^{d-1}{d\brack r+1}
(h \alpha)^{r}\pounds_{d-r}^{(h\alpha)}\bigl(g_h(\alpha)X^h\bigr)
\]
modulo $X^p-T(\alpha)$.
According to Theorem~\ref{thm:powers}, with $X$ replaced by $g_h(\alpha)X^h$ and $\alpha$ replaced by $h\alpha$,
the above expression is congruent to the desired
$
d!\,\pounds_{d}^{(h\alpha)}\bigl(g_h(\alpha)X^h\bigr),
$
but modulo $\bigl(g_h(\alpha)X^h\bigr)^p-T(h\alpha)$.
However, this polynomial is a multiple of the desired modulus $X^p-T(\alpha)$
because
$
T(h\alpha)=
g_h(\alpha)^p\,
T(\alpha)^h.
$
This can be seen by setting $X=\alpha-\alpha^{1/p}$ in the congruence
\[
g_h(\alpha)\bigl(L_{p-1}^{(\alpha)}(X)\bigr)^h
\equiv L_{p-1}^{(h\alpha)}(hX)
\pmod{X^p-(\alpha^p-\alpha)},
\]
which is~\cite[Equation~6]{AviMat:glog},
and then taking $p$th powers of both sides.
\end{proof}

We conclude the paper with a proof of Theorem~\ref{thm:GP}, which also uses Lemma~\ref{lemma:G_0G_1}.

\begin{proof}[Proof of Theorem~\ref{thm:GP}]
Expanding the left-hand side of the claimed equation we find
\[
\sum_{r=1}^{p-1} \pounds_{d}^{(r\alpha^p)}\bigl(T(r\alpha)\,X\bigr)=
\sum_{k=1}^{p-1}
\biggl(
\sum_{r=1}^{p-1}
g_k(r\alpha^p)\,
T(r\alpha)^k
\biggr)
X^k/k^d.
\]
As mentioned in the proof of Theorem~\ref{thm:powers} we have
$
T(h\alpha)=
g_h(\alpha^p)\,
T(\alpha)^h
$
for $0<h<p$, whence
$
g_k(r\alpha^p)\,
T(r\alpha)^k
=T(kr\alpha)=
g_r(k\alpha^p)\,
T(k\alpha)^r.
$
Consequently, we have
\[
\sum_{r=1}^{p-1}
g_k(r\alpha^p)\,
T(r\alpha)^k
=
\sum_{r=1}^{p-1} g_r(k\alpha^p)\,
T(k\alpha)^r
=\pounds_{0}^{(k\alpha^p)}\bigl(T(k\alpha)\bigr).
\]
Computing this reduces to computing
$\pounds_{1}^{(\alpha^p)}\bigl(T(\alpha)\bigr)$
by means of Lemma~\ref{lemma:G_0G_1}.
In fact, after taking $p$th powers of both sides the congruence of Lemma~\ref{lemma:G_0G_1} yields
\[
\pounds_{0}^{(\alpha^p)}(X^p) \cdot \pounds_{1}^{(\alpha^p)}(X^p)
\equiv
-\pounds_{1}^{(\alpha^p)}(X^p)
-\alpha \pounds_{0}^{(\alpha^p)}(X^p)
\pmod{X^p-T(\alpha)},
\]
whence
\[
\pounds_{0}^{(\alpha^p)}\bigl(T(\alpha)\bigr) \cdot \pounds_{1}^{(\alpha^p)}\bigl(T(\alpha)\bigr)
=
-\pounds_{1}^{(\alpha^p)}\bigl(T(\alpha)\bigr)
-\alpha \pounds_{0}^{(\alpha^p)}\bigl(T(\alpha)\bigr)
\]
in $\F_p[\alpha]$.
Now taking $p$th powers of both sides of the congruence
$-\pounds_{1}^{(\alpha)}\bigl(L_{p-1}^{(\alpha)}(X)\bigr)\equiv X
\pmod{X^p-(\alpha^p-\alpha)}$
and then replacing $X^p$ with $\alpha^p-\alpha$ we find $\pounds_{1}^{(\alpha^p)}\bigl(T(\alpha)\bigr)= \alpha-\alpha^p$.
Consequently, we find
$\pounds_{0}^{(\alpha^p)}\bigl(T(\alpha)\bigr)=\alpha^{p-1}-1$,
whence
$
\pounds_{0}^{(k\alpha^p)}\bigl(T(k\alpha)\bigr)=\alpha^{p-1}-1
$,
as desired.
\end{proof}

\bibliography{References}

\def\cprime{$'$} \def\polhk#1{\setbox0=\hbox{#1}{\ooalign{\hidewidth
  \lower1.5ex\hbox{`}\hidewidth\crcr\unhbox0}}}
\providecommand{\bysame}{\leavevmode\hbox to3em{\hrulefill}\thinspace}
\providecommand{\MR}{\relax\ifhmode\unskip\space\fi MR }
\providecommand{\MRhref}[2]{%
  \href{http://www.ams.org/mathscinet-getitem?mr=#1}{#2}
}
\providecommand{\href}[2]{#2}
\begin{thebibliography}{AM15b}

\bibitem[AM]{AviMat:glog}
Marina Avitabile and Sandro Mattarei, \emph{A generalized truncated logarithm},
  Aequationes Math., to appear, {\sf http://arxiv.org/abs/1803.11066}.

\bibitem[AM15a]{AviMat:gradings}
\bysame, \emph{Grading switching for modular non-associative algebras}, Lie
  algebras and related topics, Contemp. Math., vol. 652, Amer. Math. Soc.,
  Providence, RI, 2015, pp.~1--14. \MR{3453046}

\bibitem[AM15b]{AviMat:Laguerre}
\bysame, \emph{Laguerre polynomials of derivations}, Israel J. Math.
  \textbf{205} (2015), no.~1, 109--126. \MR{3314584}

\bibitem[EVG02]{EVG:polyanalogsI}
Philippe Elbaz-Vincent and Herbert Gangl, \emph{On poly(ana)logs. {I}},
  Compositio Math. \textbf{130} (2002), no.~2, 161--210. \MR{1883818
  (2002m:11059)}

\bibitem[Kon02]{Kontsevich}
Maxim Kontsevich, \emph{The {$1\frac12$}-logarithm. {A}ppendix to: ``{O}n
  poly(ana)logs. {I}'' [{C}ompositio {M}ath {\bf 130} (2002), no. 2, 161--210;
  {MR}1883818 (2002m:11059)] by {P}. {E}lbaz-{V}incent and {H}. {G}angl},
  Compositio Math. \textbf{130} (2002), no.~2, 211--214. \MR{1884238
  (2002m:11060)}

\bibitem[Mat05]{Mat:Artin-Hasse}
Sandro Mattarei, \emph{Artin-{H}asse exponentials of derivations}, J. Algebra
  \textbf{294} (2005), no.~1, 1--18. \MR{2171626}

\bibitem[Mat06]{Mat:exponential}
\bysame, \emph{Exponential functions in prime characteristic}, Aequationes
  Math. \textbf{71} (2006), no.~3, 311--317. \MR{2236408 (2007b:39056)}

\bibitem[Mir05]{Mirimanoff}
Dmitry Mirimanoff, \emph{L'\'equation ind\'etermin\'ee
  {$x^\ell+y^\ell+z^\ell=0$} et le crit\'erium de {K}ummer}, J. Reine Angew.
  Math. \textbf{128} (1905), 45--68. \MR{1580644}

\bibitem[MT13]{MatTau:polylog}
Sandro Mattarei and Roberto Tauraso, \emph{Congruences for central binomial
  sums and finite polylogarithms}, J. Number Theory \textbf{133} (2013), no.~1,
  131--157. \MR{2981405}

\bibitem[MT18]{MatTau:truncation}
\bysame, \emph{From generating series to polynomial congruences}, J. Number
  Theory \textbf{182} (2018), 179--205. \MR{3703936}

\bibitem[Rib79]{Ribenboim:13}
Paulo Ribenboim, \emph{13 lectures on {F}ermat's last theorem},
  Springer-Verlag, New York, 1979. \MR{551363 (81f:10023)}

\end{thebibliography}

\end{document}